\documentclass[12pt]{article}

\usepackage[utf8]{inputenc}
\usepackage{paralist}
\usepackage[hyphens]{url}   
\usepackage[colorlinks=true,linkcolor=blue,urlcolor=red,citecolor=blue]{hyperref}

\usepackage[toc,page]{appendix}

\usepackage{amsfonts,amsmath,amssymb,amsthm}
\usepackage{geometry}
\usepackage{caption}
\usepackage{subcaption}
\usepackage{marginnote}

\usepackage[dvips]{graphicx}
\graphicspath{{figures/}}
\usepackage{listings,color,pgf}

\usepackage{pgfplots}
\usepackage{tikz}
\usepgfplotslibrary{external}
\newlength\figureheight  
\newlength\figurewidth   

\usepackage{pgfplots}
\usepackage{multirow}

\definecolor{colKeys}{rgb}{0,0,1} 
\definecolor{colIdentifier}{rgb}{0,0,0} 
\definecolor{colComments}{rgb}{0,1,0.3} 
\definecolor{colString}{rgb}{0,0.5,0} 

\definecolor{dkgreen}{rgb}{0,0.6,0} 
\definecolor{gray}{rgb}{0.5,0.5,0.5} 
\definecolor{lightgray}{rgb}{0.9,0.9,0.9} 

\usepackage{float}   
\newfloat{Code}{H}{myc}
   
\usepackage{todonotes}

\newcommand{\eps}{\varepsilon}
\newcommand{\R}{\mathbb{R}}
\newcommand{\Z}{\mathbb{Z}}
\newcommand{\cP}{\mathcal{P}}
\newcommand{\cDe}{\mathcal{D}^\eps}
\newcommand{\cPe}{\mathcal{P}^\eps}
\newcommand{\cF}{\mathcal{F}}
\newcommand{\bk}{\boldsymbol{k}}

\newcommand{\divg}{\mathop{\mathrm{div}}}

\newtheorem{lemma}{Lemma}
\newtheorem{remark}{Remark}

\title{Computing Coherent Sets\\ using the Fokker-Planck Equation}

\author{Andreas Denner \and Oliver Junge \and Daniel~Matthes\footnote{Center for Mathematics, Technische Universit\"at M\"unchen, 85747 Garching bei M\"unchen}}

\date{\today}

\begin{document}

\maketitle   

\begin{abstract}
We perform a numerical approximation of coherent sets in finite-dimensional smooth dynamical systems by computing singular vectors of the transfer operator for a stochastically perturbed flow. This operator is obtained by solution of a discretized Fokker-Planck equation. For numerical implementation, we employ spectral collocation methods and an exponential time differentiation scheme. We experimentally compare our approach with the more classical method by Ulam that is based on discretization of the transfer operator of the unperturbed flow.

\end{abstract}

\section{Introduction}

Many fluid flows at the onset of turbulence exhibit regions which disperse slowly with time (so called \emph{coherent sets} \cite{froyland_lloyd_santi_10,FSM10}) while other regions disaggregate comparatively quickly.  Often, the boundary of a slowly dispersing region can be associated to a lower-dimensional object (a so called \emph{Lagrangian coherent structure} \cite{haller_yuan_00,haller01,haller_11}) which serves as a transport barrier for Lagrangian particles within a coherent region \cite{froyland_padberg_09}.

Based on the related concept of \emph{almost invariant} (resp.\ \emph{metastable}) sets in time-invariant dynamical systems \cite{DeJu99a,froyland-dellnitz_03}, recently a framework has been proposed for the computation of coherent sets via singular vectors of a certain smoothed transfer operator which describes the evolution of probability densities on phase space under the given dynamics \cite{Fr13a}.  Numerically, this operator can be approximated by a Galerkin \cite{Li76a,DeJu99a} method via evaluating the flow map explicitly by time integration of the vector field.  Depending on the type of approximation space chosen, a diffusion operator has either to be applied explicitly or is already implicitly present in the discretization ( via numerical diffusion).

In this manuscript, we propose to compute coherent sets by directly solving the Fokker-Planck equation instead.  More precisely, instead of computing the evolution of the basis of our approximation space under the deterministic dynamics and then applying diffusion, we directly compute the evolution of this basis under the stochastic push forward operator given by the solution operator of the Fokker-Planck equation.  This advection-diffusion equation can efficiently be discretised using spectral collocation (cf.\ also \cite{FrJuKo13a}).  In order to deal with aliasing in the case of dominating advection we use a skew symmetric form of the advection term and in order to deal with stiffness in time due to the Laplace operator we employ an exponential time differentiation \emph{(etd)} integrator. 

As a key advantage of our method we only need to sample the vector field at each time instance on a fixed grid of rather coarse resolution.  In particular, we do not need to integrate  trajectories of (Lagrangian) particles and no interpolation of the vector field to points off the grid.  

\section{Problem statement}

We consider a time-dependent ordinary differential equation $\dot{x} = b(t,x)$, $b:\R\times X\to \R^d$, on some bounded domain $X\subset \mathbb{R}^d$.  We fix some initial and final time $t_0, t_1\in\R$ and assume that the vector field $b$ is continuous and locally Lipschitz w.r.t.\ $x$ for all $t\in [t_0,t_1]$, such that the associated flow map $\Phi=\Phi(\cdot,t_0,t_1):X\to X$ is uniquely defined. In order not to obscure the key ideas we restrict to the case of $X$ being a hyperrectangle and furthermore $b$ being periodic in $x$ and divergence free, i.e.\ the flow map $\Phi$ being volume preserving.

Roughly speaking, we would like to compute a set $A_0\subset X$ which disperses slowly under the evolution of $\Phi$, i.e.\ which roughly retains its shape while being moved around by the flow $\Phi$.  A such set $A_0$ will be called \emph{coherent} \cite{FSM10}.

Inspired by the associated notion of an \emph{almost invariant} (resp.\ \emph{meta\-stable}) set in the autonomous setting \cite{DeJu99a}, we start formalizing this request by asking for a pair of sets $A_0,A_1\subset X$ such that $A_0$ will approximately be carried to $A_1$ by $\Phi$ in the sense that
\begin{equation}
\begin{split}
\kappa(A_0,A_1):=\frac{m(A_0\cap \Phi^{-1}A_1)}{m(A_0)}\approx 1,
\end{split}
\label{eq:large_rho}
\end{equation}
where $m$ denotes Lebesgue (i.e.\ volume) measure.  Evidently, with $A_1=\Phi (A_0)$ we obtain $\kappa(A_0, A_1)=1$ for any $A_0\subset X$, so this is not a well defined problem yet.  In fact, (\ref{eq:large_rho}) does not impose any condition on the geometries of the sets $A_0$ and $A_1$. In particular, the image set $A_1=\Phi(A_0)$ might be stretched and folded all over the domain $X$ -- but this is not the type of coherent set we have in mind.  Froyland \cite{Fr13a} observed that $\kappa(A_0,\Phi(A_0))$ is not close to $1$ \emph{for every set} $A_0$ any more as soon as one artificially adds some random perturbation to the dynamics (cf.\ \cite{DeJu99a,FrDeJu01a} for related ideas in the autonomous context).

\section{Computing coherent sets via diffusion at initial and final time}
\label{Ccsbdd}
More formally, Froyland proposes the following approach \cite{Fr13a}, see also \cite{froyland2014almost}. We will employ the \emph{transfer operator} (resp. \emph{Frobenius-Perron operator} or \emph{push forward}) $\cP$ associated to $\Phi$.  This operator describes how densities $u:X\to [0,\infty)$ are  evolved by $\Phi$: If a set of points in $X$ is initially distributed according to some density $u$ then after flowing these points forward with $\Phi$ they will be distributed according to $\cP u$.  In our case of $\Phi$ being volume preserving, $\mathcal{P}:L^1={L}^1(X,m)\rightarrow{L}^1$ is given by
\[
\cP u = u\circ\Phi^{-1}.
\]
In the sequel, we restrict our attention to ${L}^2=L^2(X,m)\subset L^1$ as we want to use the canonical scalar product on ${L}^2$. 
We are going to add an $\eps$-small random perturbation to the flow map $\Phi$ via complementing the action of $\cP$ by a diffusion operator $\mathcal{D}^{\eps}:{L}^2\rightarrow{L}^2$,
\begin{equation}
\mathcal{D}^{\eps}f(y)= \int_{X}\alpha^\eps(y-x)f(x)dx
\end{equation}
where  $\alpha^{\eps}:X\rightarrow [0,\infty)$ is a bounded kernel with $\int_{X} \alpha^{\eps}(x)dx=1$ and $\alpha^\eps \rightarrow \delta_0$ in a distributional sense as $\varepsilon\rightarrow 0$. Here, we use a diffusion with bounded support of radius $\eps$, namely $\alpha^\eps(x)=1_{B_\eps(x)\cap X}/m(B_\eps(x)\cap X)$. 
With $\cP$ and $\cDe$, we finally define the evolution operator \cite{Fr13a} $\cPe:L^2\to L^2$,
\begin{equation*}
\cPe := \cDe\cP\cDe,
\end{equation*}
cf.\ Figure \ref{froyland_pic}.  Note that $\cPe$ is \emph{stochastic}, i.e. positive and \ $\cPe 1_X=1_X$, where $1_X$ denotes the characteristic function on $X$, since we assume $\Phi$ to be volume preserving.  Further, with this choice of $\alpha^\varepsilon$, $\cPe$ is compact and as the vector field $b$ is divergence-free has a simple leading singular value $\sigma_1=1$ \cite{DeJu99a,Fr13a}.

\begin{figure}
\centering \def\svgwidth{350pt}
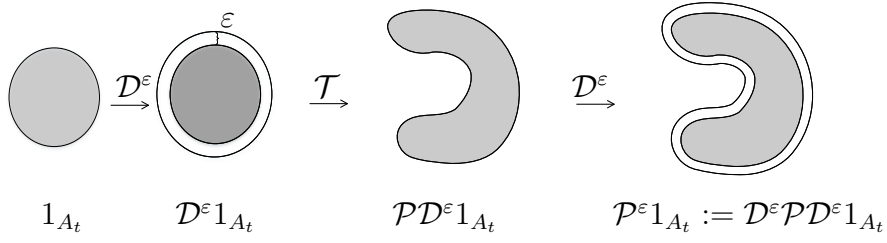 \label{picture_froyland}
\caption{Graphical illustration of the action of the perturbed transfer operator $\mathcal{P}^\eps$ as the composition of diffusion operators at initial and final time and the transfer operator $\mathcal{P}$.}
\label{froyland_pic}
\end{figure}

In order  to compute coherent sets, note that if $\kappa(A_0,A_1)\approx 1$, then $\kappa(A_0^c,A_1^c)\approx 1$ as well.  
Note further that 
$m(A_0\cap \Phi^{-1}A_1)
 =\langle 1_{A_0},1_{A_1}\circ  \Phi\rangle 
 =\langle \mathcal{P}1_{A_0},1_{A_1}\rangle$, where $\langle \cdot, \cdot \rangle$ is the standard inner product on $L^2$.
We therefore define 
\begin{align}
\label{ratio}
\rho_\eps(A_0,A_1):=\frac{\langle \mathcal{P}^\eps 1_{A_0},1_{A_1}\rangle}{m(A_0)}+   \frac{\langle \mathcal{P}^\eps 1_{A_0^c},1_{A_1^c}\rangle}{m(A_0^c)}
\end{align}
and look for some pair $(A_0,A_1)$ of sets which maximizes this quantity.
We get
\begin{align*}
\max_{A_0,A_1}\rho_\eps(A_0,{A_1})-1
&=\max_{A_0,A_1}\langle \mathcal{P}^\eps(c_0 1_{A_0}- c_0^{-1} 1_{A_0^c}), c_1 1_{A_1}-c_1^{-1} 1_{A_1^c}\rangle =: (\star),
\end{align*}
where $c_0 = \sqrt{m(A^c_0)/m(A_0)}$ and $c_1 = \sqrt{m(A^c_1)/m(A_1)}$ and since the functions $f:=c_0 1_{A_0}- c_0^{-1} 1_{A_0^c}$ and $g:= c_1 1_{A_1}-c_1^{-1} 1_{A_1^c}$ statisfy $\|f\|_2=\|g\|_2=1$ and $\langle f,1_X\rangle=\langle g,1_X\rangle=0$, we obtain, by relaxing to arbitrary functions $f,g$ with zero mean,
\begin{align}
(\star) & \leq 
\max_{\|f\|_2=\|g\|_2=1}          \left\{   {\langle \mathcal{P}^\eps f,g\rangle} : \langle f,1_X\rangle =\langle g,1_{X}\rangle=0 \right\}. \label{relaxed2}
\end{align}
This problem is much easier to solve than $(\star)$, where we need to maximize over characteristic functions.  As for fixed $f\in L^2$
\begin{align*}
\max_{\|g\|_2=1} \left\langle \mathcal{P}^\eps f,g\right\rangle = \left\langle \mathcal{P}^\eps f,\frac{\mathcal{P}^\eps f}{ \|\mathcal{P}^\eps f\|_{2}}\right\rangle 
 =\|\mathcal{P}^\eps f\|_{2}
 \end{align*}
and $\langle \mathcal{P}^\eps f,1_{X}\rangle=0$, we obtain
\begin{align*}
\label{relaxed3}
(\star) & \leq \max_{\|f\|_2=1} \left\{ \|\mathcal{P}^\eps f\|_{2} : \langle f,1_X\rangle =0 \right\} =\|\mathcal{P}^\eps\|_{L^2(V,m)},  
\end{align*}
where $V = \{f\in L^2 : \langle f,1_X\rangle=0\}$ is the orthogonal complement of $\text{span } 1_X$.

This operator norm is given by $\sigma_2(\mathcal{P}^\eps)$ \cite{Fr13a}, the second largest singular value of $\mathcal{P}^\eps$ and the maximizing function is $f=v_2$, the associated right singular function.  As $v_2$ is an approximation to $ c_0 1_{A_0}-c_0^{-1}1_{A_0^c}$ the common heuristics is to identify an associated coherent set by
\[
A_0:=\{x\in X:v_2(x) >\theta\}
\]
where $\theta\in \mathbb{R}$ is some appropriately chosen threshold \cite{DeJu99a}, \cite{froyland-dellnitz_03}. Correspondingly, $u_2=P^\eps v_2$ is the associated left singular function and $A_1=\{x\in X:u_2(x) >0\}$.

To sum up, we can compute a pair $A_0, \ A_1$ of coherent sets via computing the \emph{second singular value} and its corresponding singular vectors of a slightly perturbed Frobenius-Perron operator. Often a low order spatial discretisation such as Ulam's method is used \cite{ulam2004problems}, which is a Galerkin projection on indicator functions of a partition of $X$ into boxes $X_i$.  Usually $P$ is then computed numerically via sampling $K$ test points $x_i$ in every box $X_i$, compute $\Phi(x_i)$ and count how many fall in $B_j$:
\begin{align*}
P_{ji}=\frac{\# \Phi(x_i)\in X_j}{K}.
\end{align*}
The method automatically adds sufficient numerical diffusion so that we can actually directly employ the unperturbed operator $\mathcal{P}$.  However, with increasing resolution this numerical diffusion decreases which results in all singular values approaching the value one, cf.\ \cite{junge2004uncertainty,Fr13a}.

\section{Computing coherent sets via time-continuous diffusion}
\label{sec:FokkerPlanck}
  
Instead of explicitly applying diffusion at the beginning and the end of the time  interval under consideration, we propose to incorporate a small random perturbation continuously in time, i.e.\ instead of considering a deterministic differential equation we now use the stochastic differential equation 
\begin{equation}
\label{SDE}
dx = b(t,x)dt + \eps dB
\end{equation}
in order to define the flow map $\Phi$.  Here, $(B_t)_{t\geq 0}$ is $d$-dimensional Brownian motion and $\eps >0$.
Since we assume $b(t,\cdot)$ to be Lipschitz, $X$ to be bounded and $b$ to be periodic in $x$, for any initial condition $\xi\in X$,  \eqref{SDE} has a unique continuous solution $x$ in the sense of \cite{oksendal2003stochastic}, Thm. 5.2.1.

\begin{figure}[t]
\centering \def\svgwidth{270pt}
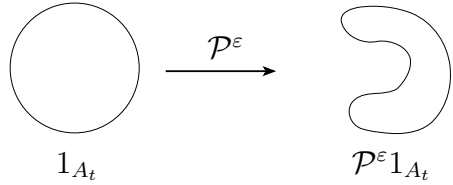 \label{picture_fokker_planck}
\caption{In contrast to the approach using diffusion at initial and final time cf.  \ref{froyland_pic}, in our approach diffusion is built into the model of the dynamical system.}
\end{figure}

The transfer operator $\mathcal{P}^\eps$ associated to this stochastic differential equation is given by the solution operator of the 
parabolic \emph{Fokker-Planck equation}
\begin{equation}
\label{FokkerPlanck}
\partial_t u = L^\eps u:=\tfrac{\eps^2}{2} \Delta u - \divg(ub).
\end{equation}
Appropriate boundary conditions are chosen (e.g., periodic or homogeneous Neumann boundary conditions), 
so that for all $u,w\in L^2(X)$ in the domain of $L^\eps$ holds:
\begin{equation}
  \label{eq:1}
  \langle w, L^\eps u\rangle = -\int_X\nabla u\cdot\nabla w + \int_X u\,b\cdot\nabla w.
\end{equation}
More precisely, $\cPe u_0 = u(t_1,\cdot)$, where $u$ is the solution to (\ref{FokkerPlanck}) with initial condition $u(t_0,\cdot)=u_0$.  

\begin{lemma}
\label{lemma:Pe_compact}
If
\[
\|b\|_{C^1}=\sup_{s\in [t_0,t_1]}\sup_{x\in X}\max\left\{|b(s,x)|, |\partial_{x_1} b(s,x)|,\ldots,|\partial_{x_d} b(s,x)|\right\}<\infty,
\]
then $\mathcal{P}^\eps:{L}^2(X)\rightarrow {L}^2(X)$ is compact.
\end{lemma}

A proof of this claim is given in the appendix.  Note that with Schauder's theorem also $\mathcal{P}^{\eps,*}$ is compact. Since we assumed $b$ to be divergence free, we also have:

\begin{lemma}
\label{lemma:Pe_stochastic}
$\mathcal{P}^\eps:L^2\to L^2$ and $\mathcal{P}^{\eps,*}:L^2\to L^2$ are stochastic.
\end{lemma}

\begin{proof}
We set $\varepsilon=\sqrt{2}$ without loss of generality. 
Since $L^\eps 1_X = \Delta 1_X-\divg(1_Xb) =\divg b=0$, 
it follows that $1_X$ is a steady state of \eqref{FokkerPlanck}, 
and consequently $\mathcal{P}^{\eps}1_X=1_X$. 

For the adjoint operator $\mathcal{P}^{\eps,*}$ we test with $u\in{L}^2(X)$:
\begin{align*}
\langle \mathcal{P}^{\eps,*}1_X,u\rangle&=\langle1_X, \mathcal{P}^\eps u\rangle
= \int_X \mathcal{P}^\eps u=\int_X u 
=\int_X u 1_X =\langle 1_X,u\rangle,
\end{align*}
where we have used that $\mathcal{P}^\eps$ is integral conserving:
\begin{align*}
  \partial_t\int_X u
  =\int_X L^\eps u= \langle 1_X,L^\eps u\rangle 
  =-\int_X \nabla u\cdot\nabla 1_X + \int_X  u\,b\cdot\nabla 1_X =0,
\end{align*}
thanks to the integration-by-parts rule \eqref{eq:1}.
Hence $\int_X u=\int_X \mathcal{P}^\eps u$ for all $u\in {L}^2(X)$, cf.\ \cite{lasota1993chaos}. 
By the Riesz representation theorem, we can conclude that also $\mathcal{P}^{\eps,*}1_X=1_X$. 

For proving positivity of $\mathcal{P}^\eps$, 
we consider the evolution of the negative part $u_-(s,x)=-\min(u(s,x), 0)$.
Since $\partial_t (u_-^2)= - 2u_-\partial_t u$ almost everywhere on $X\times(t_0,t_1)$,
we have, using the integration-by-parts rule \eqref{eq:1},
\begin{equation}
  \begin{split}
    \frac{1}{2}\partial_t\int_X u_-^2&=
    -\int_X u_-\partial_t u= -\int_X u_-L^\eps u \\
    &=\int_X \nabla u\cdot\nabla u_- - \int_X u\,b\cdot\nabla u_-
    =-\int_X |\nabla u_-|^2 + \frac{1}{2} \int_X b\cdot \nabla (u_-^2) \\
    &=-\int_X |\nabla u_-|^2 - \frac{1}{2} \int_X \divg(b) \ u_-^2\\
    &=-\int_X |\nabla u_-|^2\leq 0. 
  \end{split}
  \label{eq:normdecreasing}
\end{equation}
Hence if $u(t_0,\cdot)$ is non-negative, $u(t,\cdot)$ is non-negative for all $t>t_0$, 
as the norm of its negative part does not increase. 
To show positivity of $\mathcal{P}^{\eps,*}$, let two non-negative functions $u,w\in L^2(X)$ be given.
Then
\begin{align*}
  \langle \mathcal{P}^{\eps,*}w,u\rangle = \langle w,\mathcal{P}^\eps u\rangle \ge 0
\end{align*}
by positivity of $\mathcal{P}^\eps$.
For any fixed non-negative $w$, this relation holds for all non-negative $u$.
This implies that $\mathcal{P}^{\eps,*}w$ is non-negative.
\end{proof}

\begin{lemma}
The leading singular value $\sigma_1=1$ of $\mathcal{P}^\eps$ is isolated.
\end{lemma}
\begin{proof}
With the same argument as in \eqref{eq:normdecreasing} we obtain
for solutions to \eqref{FokkerPlanck}:
\begin{align*}
\frac{1}{2}\partial_t \int u^2\leq - \int|\nabla u|^2.
\end{align*}
This implies that $\|u(t,\cdot)\|_2$ is non-increasing in time.
Moreover, if the initial datum $u$ is not constant on $X$, its gradient does not vanish almost everywhere,
and hence $\|u(t,\cdot)\|_2$ decreases on a small time interval $(t_0,t_0+\tau)$.
Thus $\|\mathcal{P}^\eps u\|_2<\|u\|_2$, and $u$ cannot be an eigenfunction of $\mathcal{P}^\eps$ corresponding to the eigenvalue $1$. With the same argument the constant function is the only eigenfunction of $\mathcal{P}^{\eps,*}$ corresponding to the eigenvalue 1. Hence the leading singular value $\sigma_1=1$, which is an eigenvalue of $\mathcal{P}^{\eps, *}\mathcal{P}^{\eps}$ is isolated.

\end{proof}
Hence $\mathcal{P}^\eps$ is compact and doubly stochastic with isolated, simple leading singular value $\sigma_1$. So we are in a similar setting as in section \ref{Ccsbdd} and  apply the same constructions.

\begin{remark}
These considerations also work for a non-conservative vector field $b$ and  corresponding probability measure $\mu\neq m$ by suitably normalizing the operator $\cP$ resp.\ $\cP^\eps$, cf.\ \cite{Fr13a}.
\end{remark}

\section{Discretisation}

In order to approximate the transfer operator $\cP^\eps$, we choose a finite dimensional approximation space $V_N \subset {L}^2(X)$ and use collocation. As the Fokker-Planck equation \eqref{FokkerPlanck} is parabolic and since we assume the vector field $b(t,\cdot)$ to be smooth for all $t$, its solution $u(t,\cdot)$ is smooth for all $t>t_0$ (see e.g. \cite{evans2010partial}, Ch. 7, Thm. 7).  To exploit this, we choose $V_N$ as the span of the Fourier basis 
\begin{align*}
\varphi_{\bk}(x) = e^{i \langle\bk, x\rangle}, \quad \bk\in\Z^d, \|\bk\|_\infty\leq (N-1)/2, N \text{ odd}. 
\end{align*}
Note that $\text{dim}(V_N)=N^d$. Choosing a corresponding set $\{x_1,\ldots,x_M\}\subset X$ of collocation points (typically on an equidistant grid), the entries of the matrix representation $P^\eps$ of $\cPe$ are given by
\begin{align}\label{eq:Pe}
P^\eps_{j\bk}= \cPe\varphi_{\bk}(x_j), 
\end{align}
where $\bk\in\Z^d, \|\bk\|_\infty\leq (N-1)/2$ and $j=1,\ldots,M$.  For $P^\eps$, we then compute singular values and vectors via standard algorithms. Note that we might choose $M\geq N$, i.e.\ more collocation points than basis functions.  This turns out to be useful since we expect the maximally coherent sets to be comparatively coarse structures which can be captured with a small number of basis functions.

\paragraph{Solving the Fokker-Planck equation.}

In order to compute $\cPe\varphi_{\bk}$ in (\ref{eq:Pe}) for some basis function $\varphi_{\bk}\in V_N$ we need to solve the Fokker-Planck equation \eqref{FokkerPlanck} with initial condition $u(t_0,\cdot)=\varphi_{\bk}$.  This can efficiently be done  in Fourier space via integrating the Cauchy problem
\begin{align*}
\partial_t \hat u &=\tfrac{\eps^2}{2} \hat \Delta \hat u - \hat\divg \mathcal{F}(\mathcal{F}^{-1}(\hat{u}) b),\quad \hat u(t_0,\cdot)=\hat\varphi_{\bk},
\end{align*}
in time, where $\hat v = \mathcal{F}(v)$ is the Fourier transform of $v\in V_N$.  Note that the differential operators in Fourier space reduce to multiplications with diagonal matrices, while $\mathcal{F}$ and $\mathcal{F}^{-1}$ can efficiently be computed by the (inverse) fast Fourier transform.

\paragraph{Aliasing.}

One problem with this formulation is the possible occurence of aliasing. As $\hat u$ and $\hat b$ are trigonometric polynomials of degree $N$, the multiplication $\cF^{-1}(\hat u)b$ in the advection term leads to a polynomial $\mathcal{F}^{-1}(\hat u) b$ of degree $2N$ that cannot be represented in our approximation space $V_N$. The coefficients of degree $\geq N$ of this polynomial act on the coefficients of lower degree, leading to unphysical contributions in these. This shows up in high oscillations and blow ups (see \cite{boyd2001chebyshev}, Ch. 11) in the computed solution.  One way to deal with this problem is to use the advection term 
\begin{align*}
\divg \left(  b u  \right) = \frac{1}{2} \divg \left(b u\right)+ \frac{1}{2} \left( b \nabla u   \right).
\end{align*}
The spectral discretization of the left-hand side is not skew symmetric, but the discretization of the right-hand side is \cite{zang1991rotation}. This leads to purely imaginary eigenvalues of the resulting discretization matrix and hence to mass conservation.  Consequently, for the unperturbed operator $(\eps=0)$, the resulting matrix has eigenvalues on the unit circle. However, this approach has to be used carefully as even though the solution does not blow up, it might still come with a large error, e.g. small scale structures may be suppressed. If the system produces such small scale structures the grid has to be chosen fine enough to resolve them. 

\paragraph{Time integration.}

For low resolutions, the time integration of the space discretized system can be performed by a standard explicit scheme. For higher resolutions, the stiffness of the system due to the Laplacian becomes problematic and a more sophisticated method must be employed.  Here, we use the \emph{exponential time differentiation scheme} \cite{cox2002exponential} for the space discretized system. The etd-scheme separates the diffusion term $\mathcal{L}=\tfrac{\eps^2}{2}D$, where $D$ is the discretized Laplacian, from the advection term $\mathcal{N}(u,t)=- \divg \mathcal{F}(\mathcal{F}^{-1}(\hat u) b(\cdot,t))$, where $b$ is evaluated via spectral collocation. The system can hence be written as
\begin{align}
u_t=\mathcal{L}u+\mathcal{N}(u,t).
\label{etd1}
\end{align}
Via multiplying \eqref{etd1} with $e^{-t\mathcal{L}}$ and integrating from $t_0$ to $t_1$ we obtain
\begin{align}
u(t_1)=e^{h\mathcal{L}}u(t_0)+e^{h\mathcal{L}}\int_0^h e^{-\tau\mathcal{L}} \mathcal{N}(u(t_0+\tau), t_0+\tau)d\tau,
\label{etd2}
\end{align}
with $h=t_1-t_0$.  A numerical scheme is derived by approximating the integral in \eqref{etd2}, e.g.\ by a Runge-Kutta 4 type rule, resulting in scheme called \emph{etdrk4}. To this end note that $D$ is a diagonal matrix. We use the version in  \cite{kassam2005fourth}, which elegantly treats a cancellation problem occurring in a naive formulation of \emph{etdrk4} by means of a contour integral approximated by the trapezoidal rule.

\paragraph{Extraction of coherent sets}
For the extraction of coherent sets several methods can be used. For the decomposition into exactly two coherent sets a simple thresholding or an a posteriori line search can be used \cite{Fr13a,froyland-dellnitz_03}, for the decompositition of the domain into $n$ coherent sets the first $n$ singular vectors should be considered (see \cite{DeJu99a, schutte1999direct, huisinga2006metastability} for the autonomous case) and post processed via a simple clustering heuristic, e.g. k-means, see \cite{banisch2016understanding, hadjighasem2015spectral}. We here focus on the computation of singular vectors.

\section{Numerical Experiments}

\subsection{Quadruple gyre}

The first numerical example is a two dimensional flow (cf.\ Fig.~\ref{vectorfield}), an extension of the well known double gyre flow, given by
\begin{align*}
\label{quadruple gyre}
\dot{x}&=-g(t,x,y)    \\
\dot{y}&= g(t,y,x)  
\end{align*}
on the $2$-torus $[0,2]\times [0,2]$, where
\[
g(t,x,y) = \pi\sin(\pi f(t,x))\cos (\pi f(t,y))\partial_x f(t,y)
\]
and $f(t,x)=\delta \sin (\omega t)x^2+(1-2\delta \sin(\omega t))x$. We fix $ \delta=0.25$,  $\omega=2\pi$, $t_0=0$, $t_1=10.25$, $h = 0.205$ (i.e.\ $50$ time steps) and choose $\eps=0.02$ in such a way that the six largest singular values of $P^\eps$ roughly equal those  obtained from Ulam's method (without explicit diffusion).

\begin{figure}
\begin{center}
\includegraphics[width=0.25\textwidth]{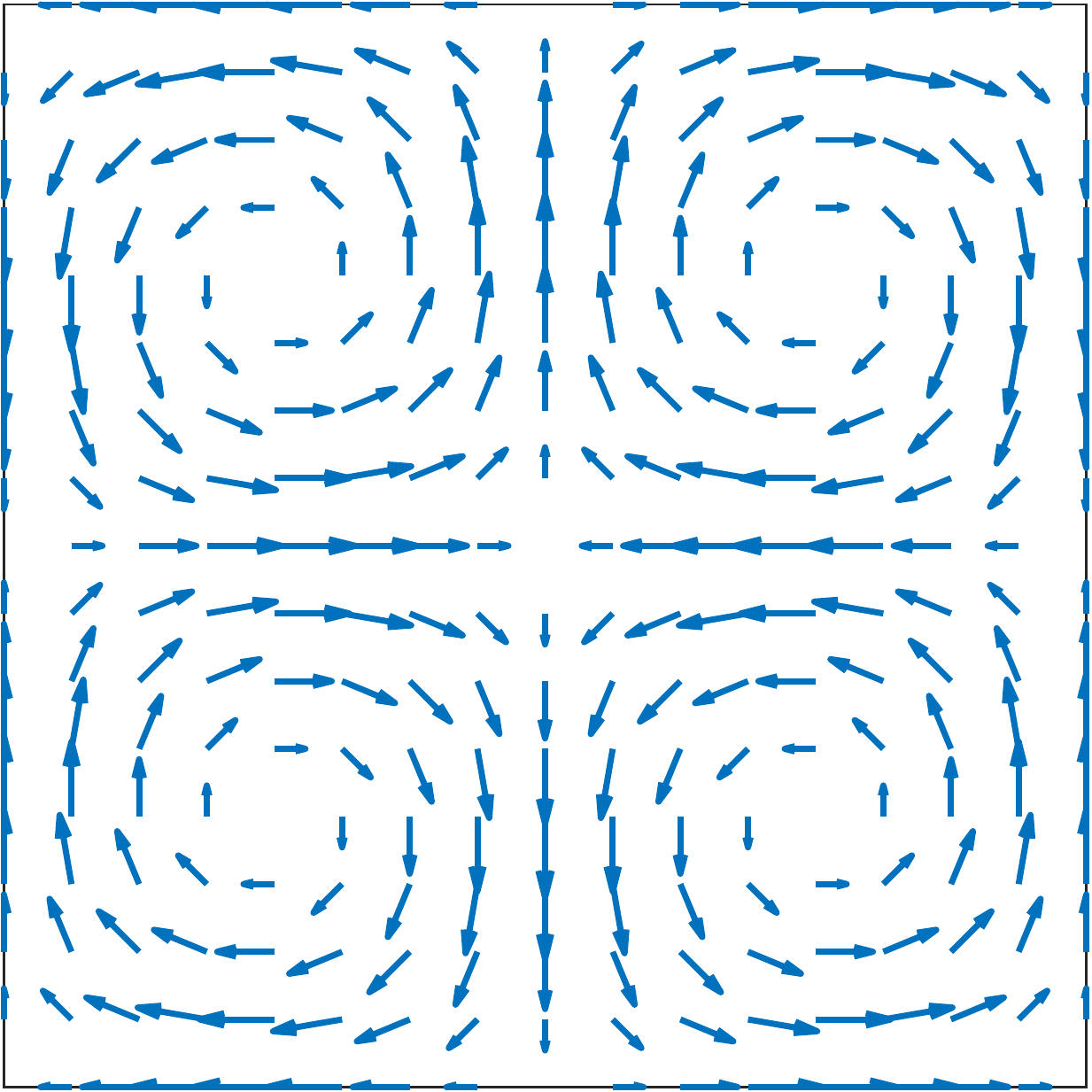}
\quad
\includegraphics[width=0.25\textwidth]{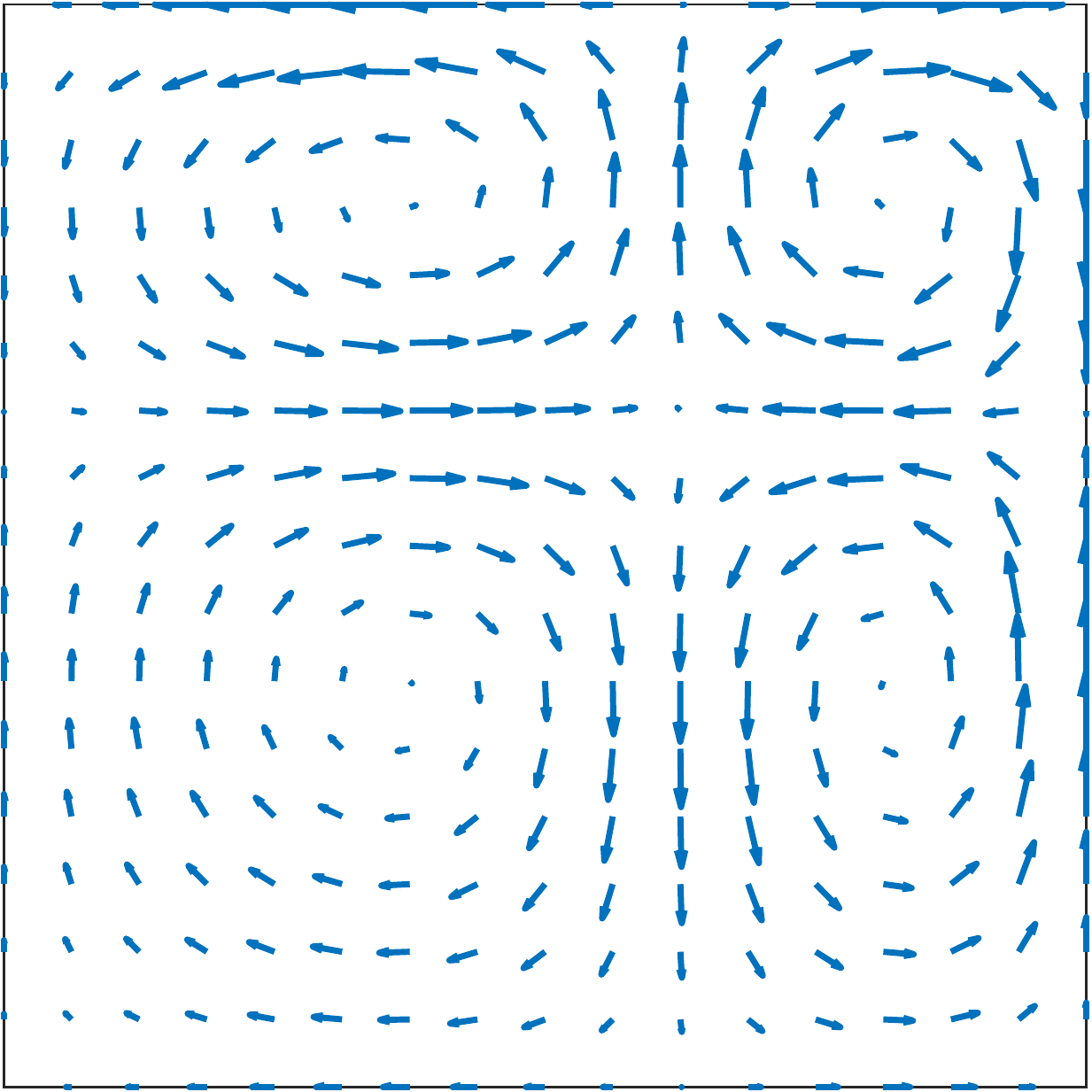}
\end{center}
\caption{Quadruple gyre vector field at $t=0$ and $t=10.25$. The four gyres are separated by a horizontal and a vertical line, such that their intersection point moves on the diagonal.}
\label{vectorfield}
\end{figure}	

We use $M=15$ collocation points and $N=5$ basis functions in each direction and compute the first four singular values and right singular vectors of $P^\eps$.  As shown in Fig.~\ref{lsv3515} (top row), they nicely reveal the gyres in their sign structures. The computation of $P^\eps$ takes less than a second, the computation of all singular values and -vectors less than $0.01$ seconds\footnote{Computation times are measured on an $2.6$ GHz Core i5 running Matlab R2015b.}.  For comparison, in the bottom row of Fig.~\ref{lsv3515}, we show the same singular vectors computed via Ulam's method (without explicit diffusion) on a $32 \times 32$ box grid using $100$ sample points per box. Here, the computation of the transition matrix takes around 5 seconds, the computation of the six largest singular values resp.\ vectors less than 0.2 seconds.

\begin{figure*}[tb]
\begin{center}
\begin{tabular}{ccccl}
\includegraphics{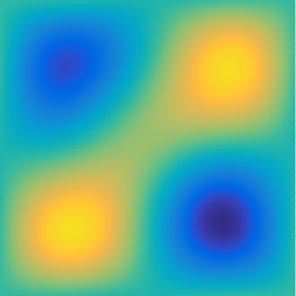} 
&
\tikzsetnextfilename{15_5_eps=001_right_3}
\includegraphics{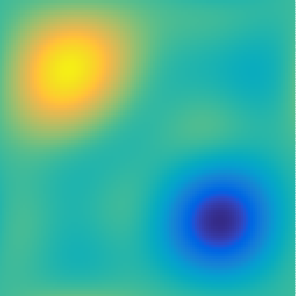} 
&
\tikzsetnextfilename{15_5_eps=001_right_4}
\includegraphics{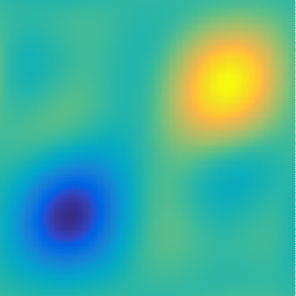} 
&
\tikzsetnextfilename{15_5_eps=001_right_5}
\includegraphics{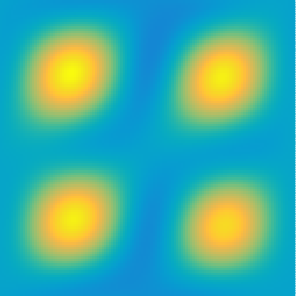} 
&
\raisebox{-0.5cm}{

\tikzsetnextfilename{colorbar2}
\includegraphics{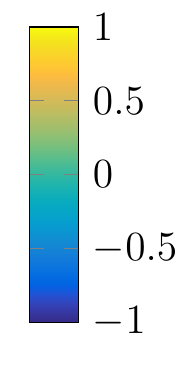} 
}

\\
$\sigma_2= 0.999$
&
$\sigma_2= 0.997$
&
$\sigma_2= 0.996$
&
$\sigma_2= 0.995$
&
\\
\tikzsetnextfilename{Ulam_32_left_2}
\includegraphics{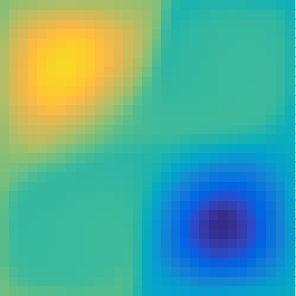} 
&
\tikzsetnextfilename{Ulam_32_left_3}
\includegraphics{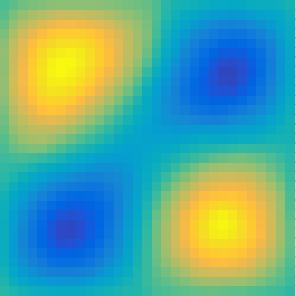} 
&
\tikzsetnextfilename{Ulam_32_left_4}
\includegraphics{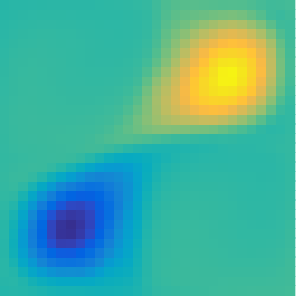} 
&
\tikzsetnextfilename{Ulam_32_left_5}
\includegraphics{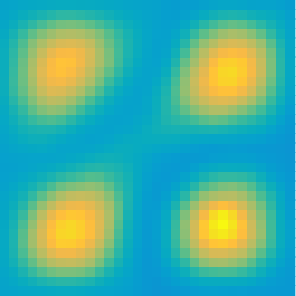} 
&
\raisebox{-0.5cm}{
\tikzsetnextfilename{colorbar2}	
\includegraphics{TikzPictures/colorbar2} 
}
\\
$\sigma_2= 0.996$
&
$\sigma_2= 0.994$
&
$\sigma_2= 0.991$
&
$\sigma_2= 0.985$
&

\end{tabular}
\end{center}
\caption{Top row: 2-th to 5-th right singular vector computed via the Fokker-Planck approach. Bottom row: computed via Ulam's method. }
\label{lsv3515}
\end{figure*}

\subsection{A turbulent flow}

We now turn to a case where the vector field is only given on a discrete grid. Here, the approach proposed in Section~\ref{sec:FokkerPlanck} is particularly appealing if we chose the grid points as collocation points (resp.\ a subset of them):  In contrast to methods which are based on an explicit integration of individual trajectories (as Ulam's method), no further interpolation of the vector field is necessary. Furthermore, depending on the initial point of a trajectory, the small scale structure of the turbulent vector field might enforce very small step sizes of the time integrator and hence makes the computation expensive. 

For the experiment, we consider the incompressible Navier-Stokes equation with constant density on the 2-torus $X=[0,2\pi]^2$,
\begin{align*}
\frac{\partial\mathbf{v}}{\partial t} 
 & =  \nu \Delta \mathbf{v} - (\mathbf{v} \cdot \nabla) \mathbf{v} - \nabla p \\
 \nabla \mathbf{v} & =0,
\end{align*}
where $\mathbf{v}$ denotes the velocity field, $p$ the pressure, and $\nu > 0$ the viscosity. Via introducing the vorticity $w:=\nabla \times \mathbf{v}$, the Navier-Stokes equation in 2D can be rewritten as \emph{vorticity equation} 
\begin{eqnarray}\label{eq:vorticity}
\frac{Dw}{Dt}=\partial_t w+ (u\nabla) w &=& \nu \Delta w\\
\nonumber \Delta\psi &=& -w.
\end{eqnarray}
where the pressure $p$ cancels from the equation. We can extract the velocity field $\mathbf{v}$ from the \emph{streamline function} $\psi$ via $\mathbf{v}_1=\partial_y \psi$ and $\mathbf{v}_2=-\partial_x \psi$.

%
The equation can be integrated by standard methods, e.g.\ a pseudo spectral method as proposed in \cite{MIT_NavierStokes}. For a first experiment, we choose an initial condition inducing three vortices, two with positive and one with negative spin,  as initial condition:
\begin{align*}
w(0,x,y)=
 e^{-5\|(x,y)-(\pi,\frac{\pi}{4})\|_2^2}
+e^{-5\|(x,y)-(\pi,-\frac{\pi}{4})\|_2^2} -\frac{1}{2}
 e^{-\frac{5}{2}\|(x,y)-(\frac{\pi}{4},\frac{\pi}{4})\|_2^2}.
\end{align*}
We solve \eqref{eq:vorticity} on a grid with 64 collocation points in both coordinate direction. For the computation of coherent sets we chose $n=16$ basis functions and $N=32$ collocation points in both directions, as well as $t_0=0$ and $t_1=20$. We hence use only every second collocation point of the computed vector field. However, as the induced coherent structures are way bigger than the grid size, this does not affect the result. The underresolution of the vector field can be interpreted as additional diffusion. We use $\varepsilon=10^{-2}$ which is of the same order as the grid resolution. In Fig.~\ref{vortices_f0}, we show the vector field at time $t_0=0$ (left) as well as the second right singular vector (center) in the first row as well as the vector field at $t_1=20$ and the corresponding second left singular vector (center) in the second row. The computation took 35 seconds.

For comparison, we show the same singular vectors computed via Ulam's method (right) on a $32\times 32$ box grid using 100 sample points per box which were integrated by Matlab's \texttt{ode45}.  Here, we need to interpolate the vector field between the grid points using splines (i.e.\ using \texttt{interp2} in Matlab).  This computation also took 35 seconds.

\begin{figure*}[tb]
\begin{center}
\begin{tabular}{ccc}
\includegraphics[width=4cm]{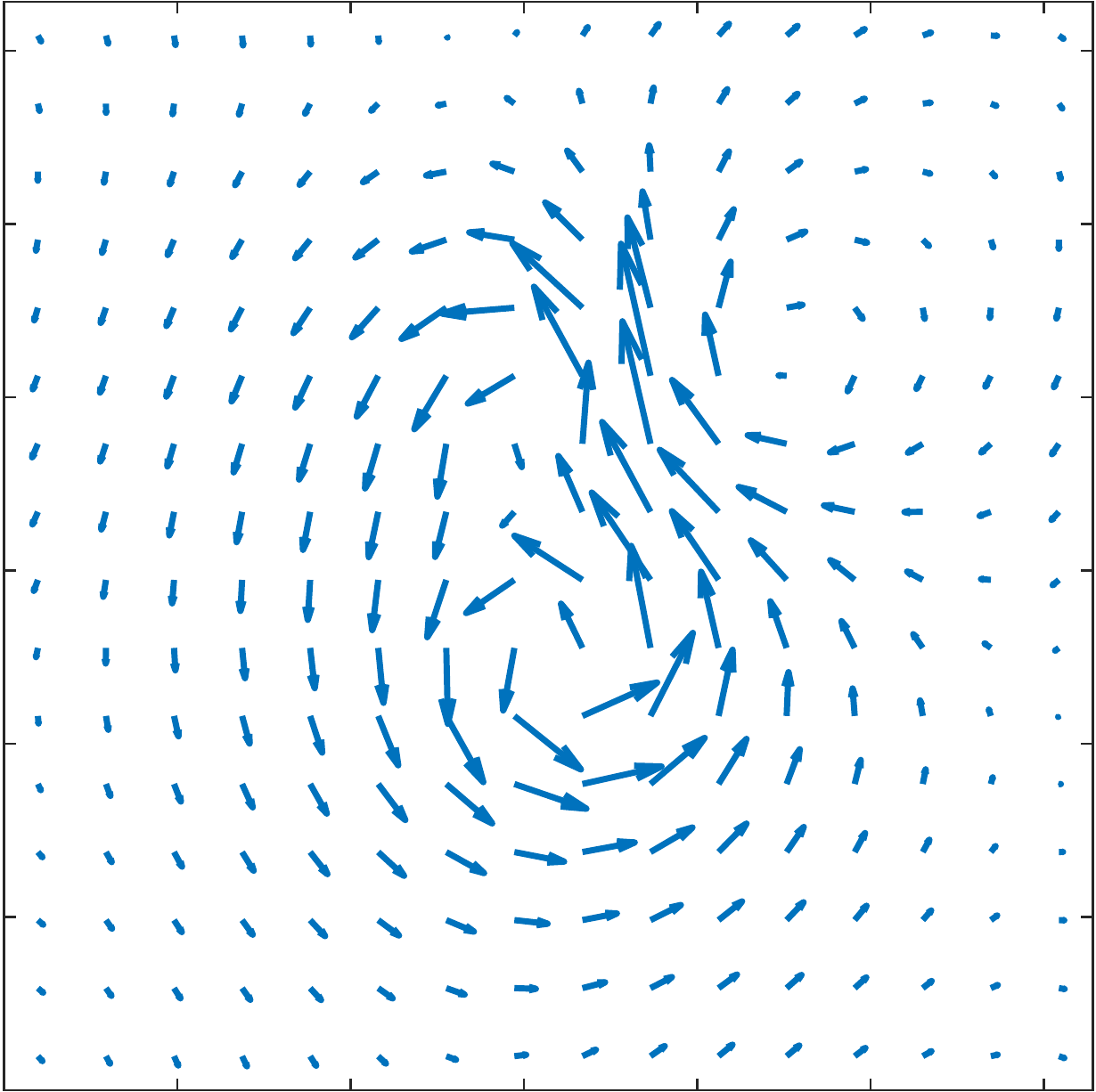}
&
\tikzsetnextfilename{32_16_eps=001_right_2}
\includegraphics{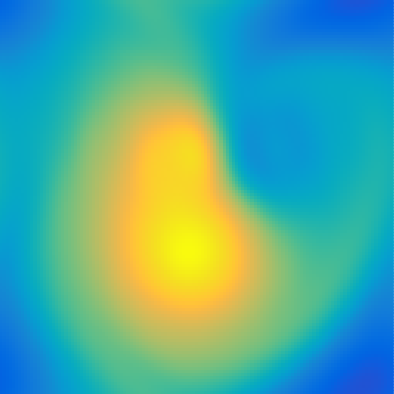} 
&
\tikzsetnextfilename{Ulam_32_left_22}
\includegraphics{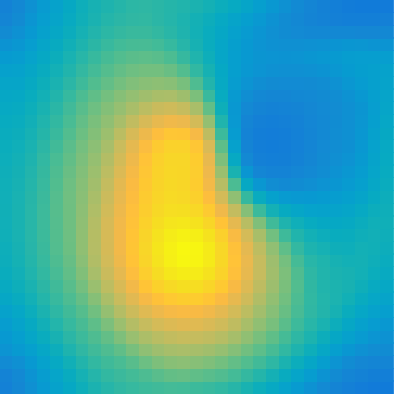} 
\\
\includegraphics[width=4cm]{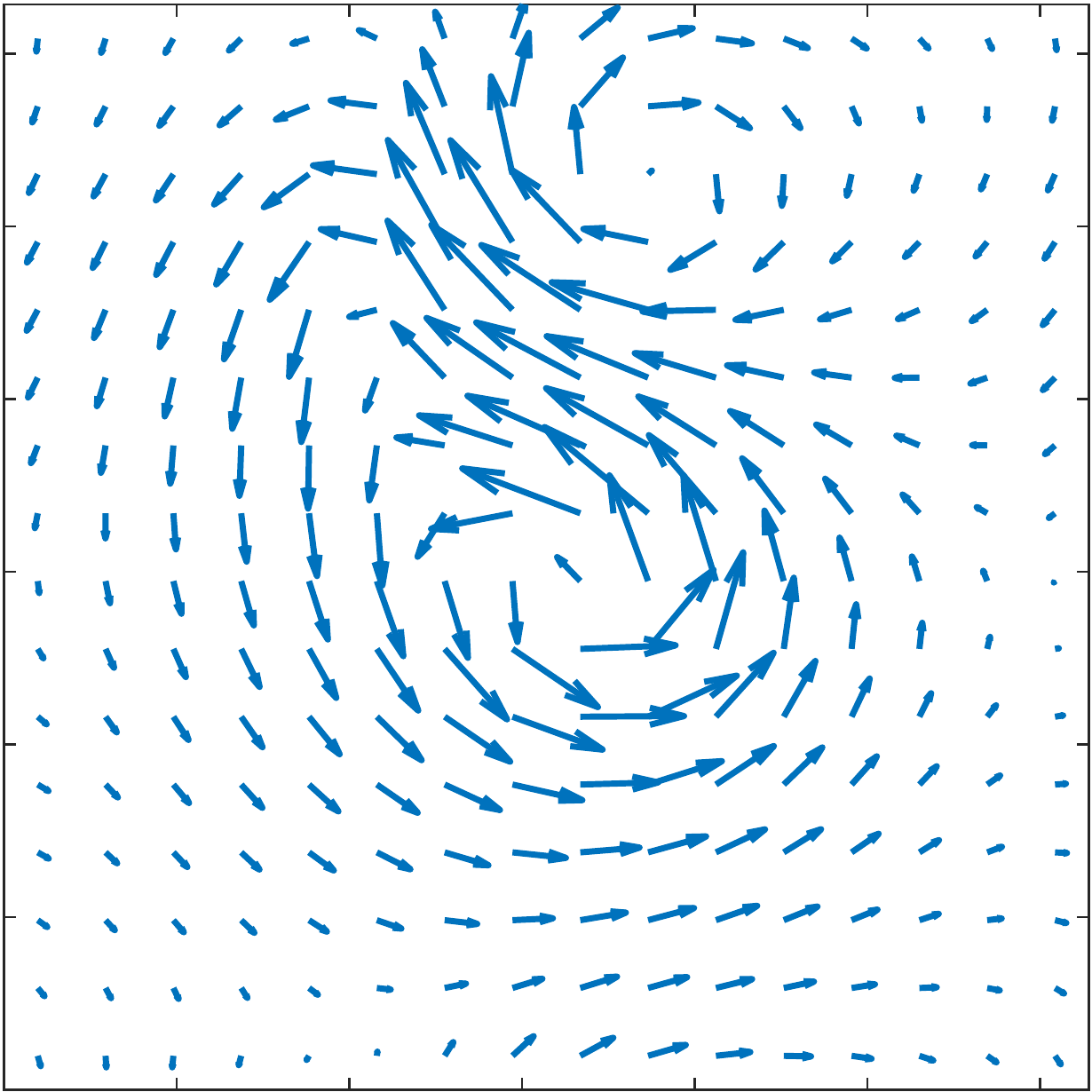}
&
\tikzsetnextfilename{32_16_eps=001_left_2}
\includegraphics{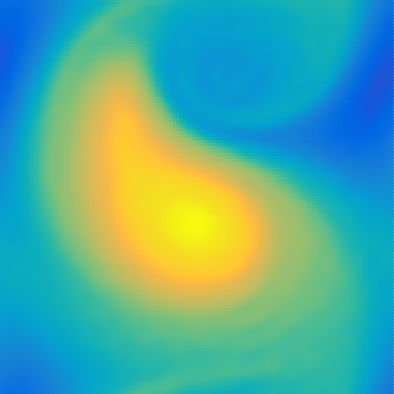} 
&
\tikzsetnextfilename{Ulam_32_right_22}
\includegraphics{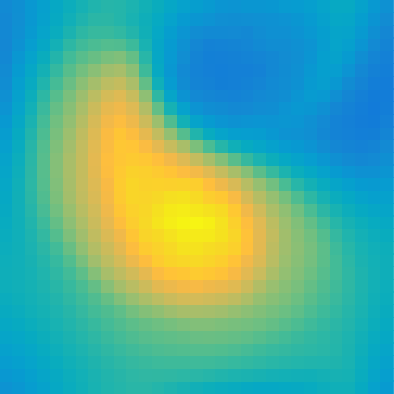} 
\end{tabular}
\end{center}
\caption{Top row: vector field at $t_0=0$ (left), second right singular value computed via Fokker-Planck (center) and via Ulam (right). Bottom row: vector field at $t_1=20$ (left), second left singular vector computed via Fokker-Planck (center) and via Ulam (right).}
\label{vortices_f0}
\end{figure*}

For a second experiment, we use a turbulent initial condition by choosing a real number randomly in $[-1,1]$ from a uniform distribution at each collocation point.  For the coherent set computation, we restrict the time domain to $[t_0,t_1]=[20,40]$ since then the initial vector field has smoothed somewhat, cf.\ Fig.~\ref{turbulent_10}. Here, we choose $M=64$, i.e.\ more collocation points than in the examples before as the vector field lives at smaller scales, $N=16$ and $\varepsilon=10^{-3}$ which is of the same order as $\nu$.  The results are non-obvious pairs of coherent sets, cf.\ Fig.~\ref{turbulent_10}. The maximally coherent set indicated by the second singular vectors describes the vortex in the upper left region. The computation took 100 seconds.

For comparison, we show the same singular vectors computed via Ulam's method on a $32\times 32$ box grid using 25 sample points per box. Here, we need to interpolate the vector field between the grid points using splines (i.e.\ using \texttt{interp2} in Matlab). Since the vector field is turbulent, using Matlab's \texttt{ode45} for the vectorized system is infeasible.  We therefore choose a fixed time-step of $h=0.01$, such that the result does not seem to change when further decreasing $h$.  This computation also took roughly 100 seconds.

\begin{figure*}[tb]
\begin{center}
\begin{tabular}{cccc}
\includegraphics[width=3cm]{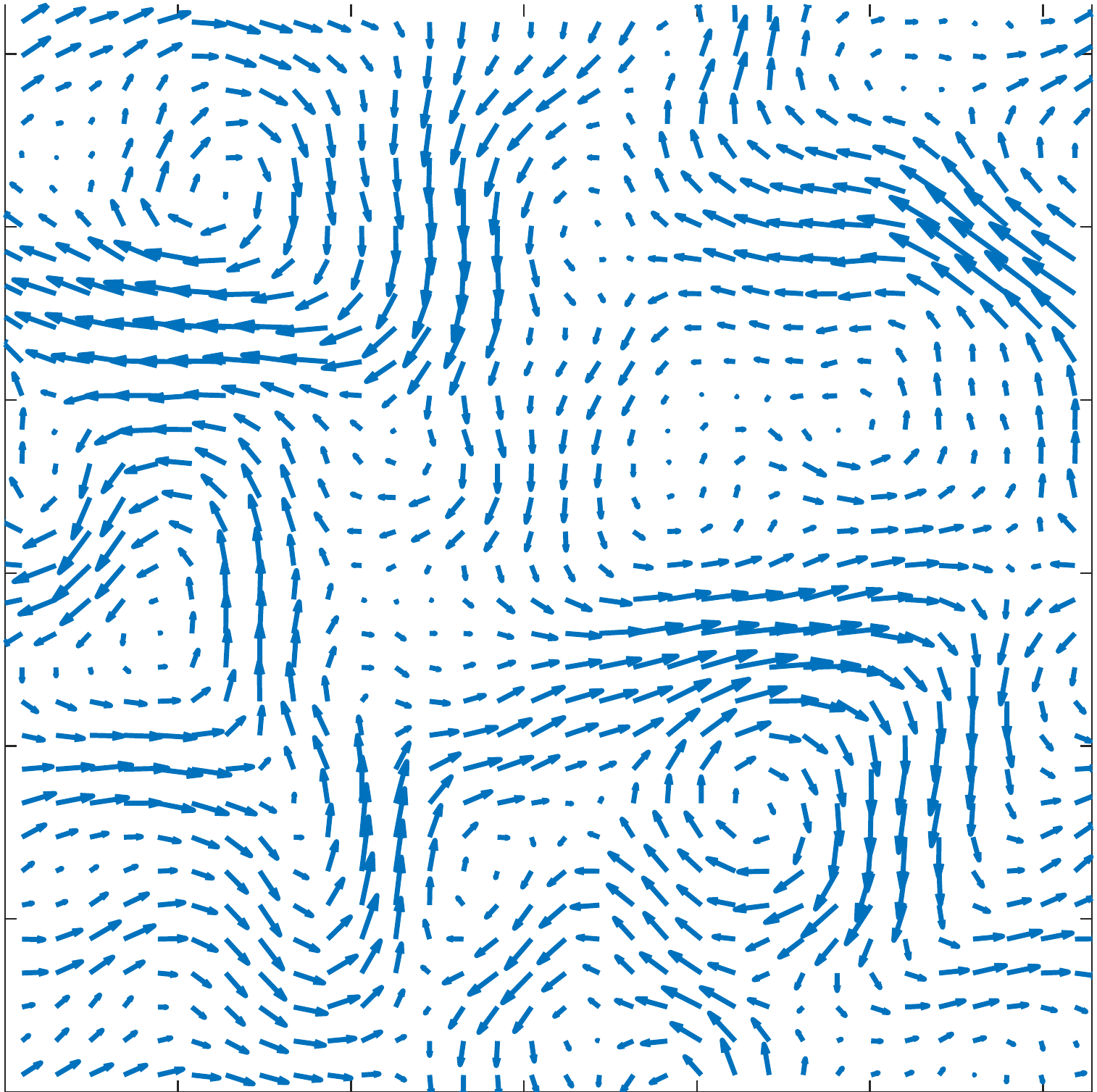}
&
\tikzsetnextfilename{64_16_eps=001_right}
\includegraphics{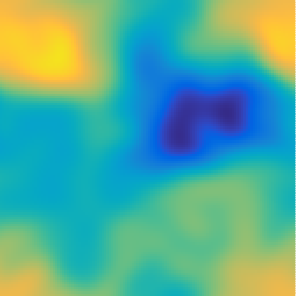} 
&
\tikzsetnextfilename{Ulam_32_left_23}
\includegraphics{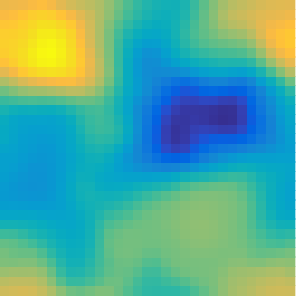} 
&
\tikzsetnextfilename{uncoherent_t0}
\includegraphics{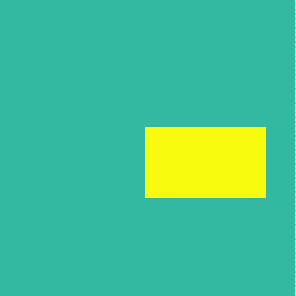} 
\\
\includegraphics[width=3cm]{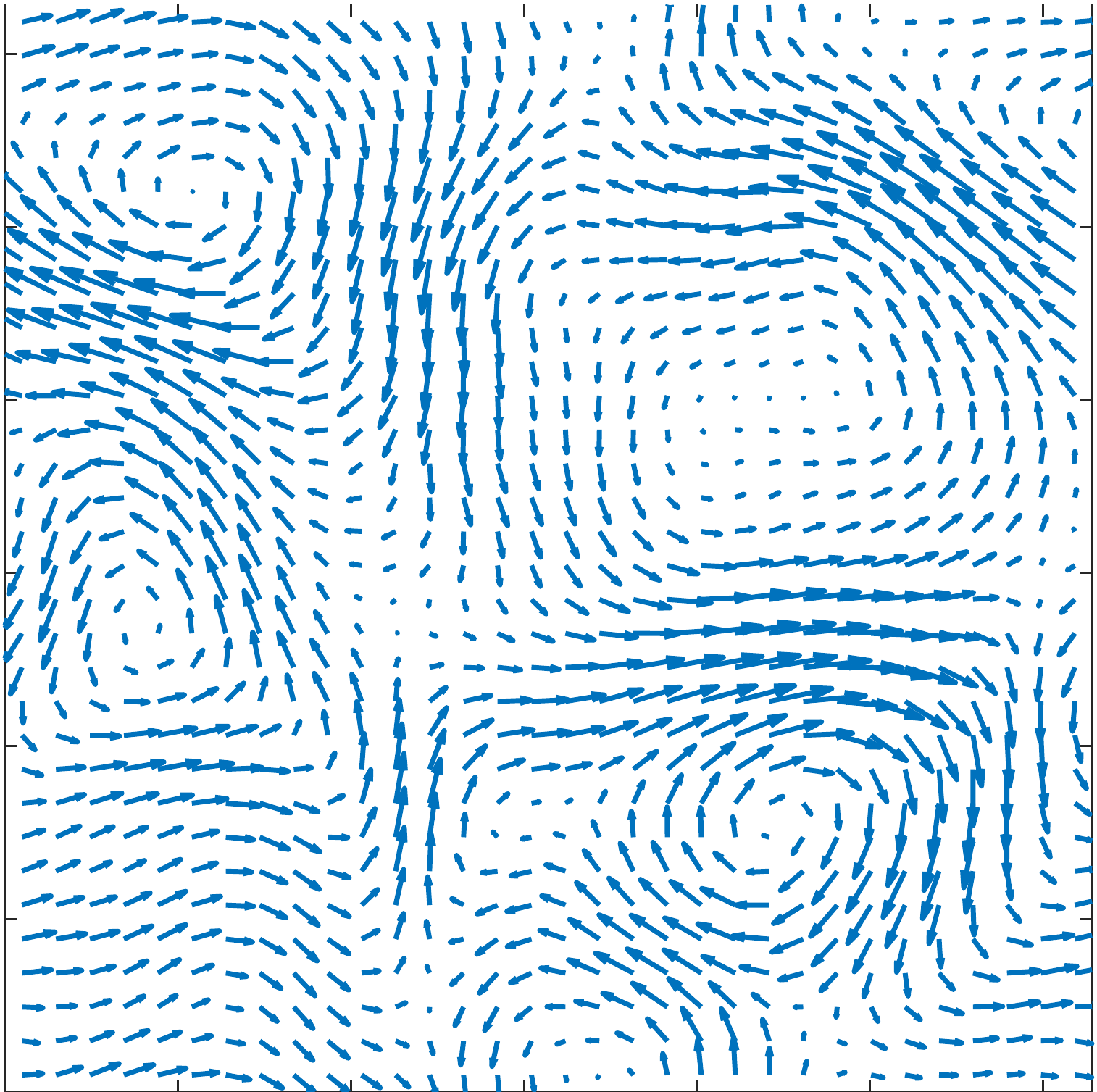}
&
\tikzsetnextfilename{32_16_eps=001_left}
\includegraphics{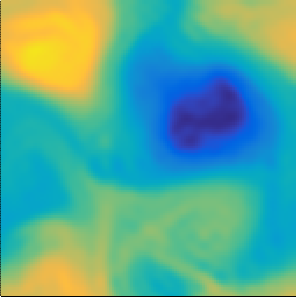} 
&
\tikzsetnextfilename{Ulam_32_right_23}
\includegraphics{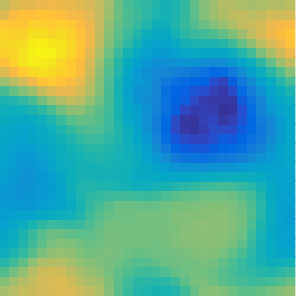} 
&
\tikzsetnextfilename{uncoherent_tf}
\includegraphics{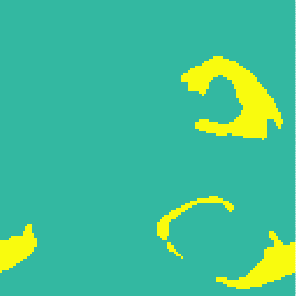} 

\end{tabular}
\end{center}
\caption{Top row: vector field at $t_0=20$ (left), second right singular vector computed via Fokker-Planck (center left) and via Ulam (center right) and an incoherent set (right). Bottom row: vector field at $t_1=40$ (left), second left singular vector computed via Fokker-Planck (center left) and via Ulam (center right) and the evolution of the incoherent set (right).}
\label{turbulent_10}
\end{figure*}

\subsection{Octuple gyre}

Finally, based on the quadruple gyre, we construct the following flow in $\R^3$ with eight gyres, given by the equations 
\begin{align*}
\label{ninefold gyre}
\dot{x}&= g(t,x,y) - g(t,x,z) \\
\dot{y}&= g(t,y,y) - g(t,y,x) \\
\dot{z}&= g(t,z,x) - g(t,z,y)
\end{align*}
on the 3-torus $X=[0,2]^3$ with $g$ and the other parameters from Section 6.1. By construction the dynamics of this system exhibits eight gyres in each quadrant of the cube $[0,2]^3$. The cross sections of this vector field at, e.g.\ $\{x=0.5\}$ and $\{x=1.5\}$ are  again given by Fig.~\ref{vectorfield}.  

In Fig.~\ref{cs3D23} we show the second to fourth left singular vectors, computed  using 16 collocation points and 4 basis functions in each direction with $\eps=0.1$. Interestingly, none of the obvious gyres is identified by  the second and the third singular vectors, but two coherent sets with `centers' at $[1,1,1]$ and $[2,2,2]$. This is unexpected as this set is not encoded in the vector field on purpose.  Starting with the fourth singular vector, the gyre centers are identified as coherent.  In Fig.~\ref{cs3D23left} we show the corresponding right singular vectors at time $t_1=10.25$. The computation time is $30$ seconds.

\begin{figure}
\begin{minipage}{0.32\textwidth}
\centering \def\svgwidth{170pt}
\includegraphics[width=\textwidth]{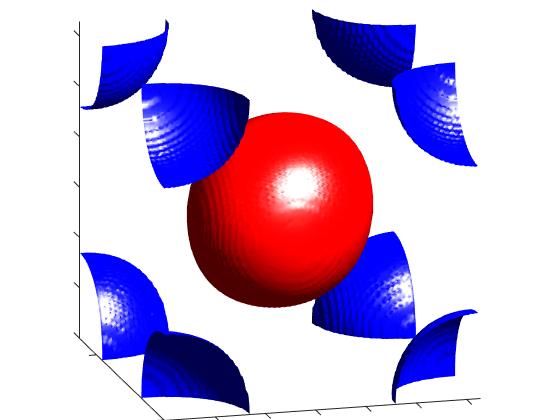}
\end{minipage}
\begin{minipage}{0.32\textwidth}
\centering \def\svgwidth{170pt}
\includegraphics[width=\textwidth]{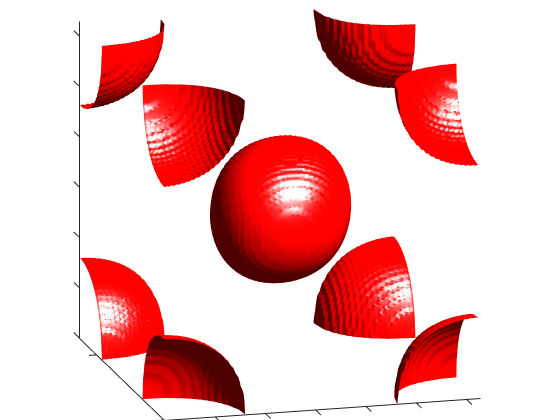}
\end{minipage}
\begin{minipage}{0.32\textwidth}
\centering \def\svgwidth{170pt}
\includegraphics[width=\textwidth]{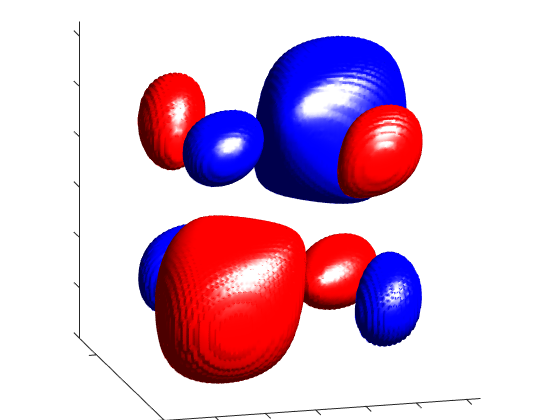}
\end{minipage}
\caption{Octuple gyre: Second to fourth right singular vectors at time $t_0=0$ (red positive, blue negative level set).}
\label{cs3D23}
\end{figure}

\begin{figure}
\begin{minipage}{0.32\textwidth}
\centering \def\svgwidth{170pt}
\includegraphics[width=\textwidth]{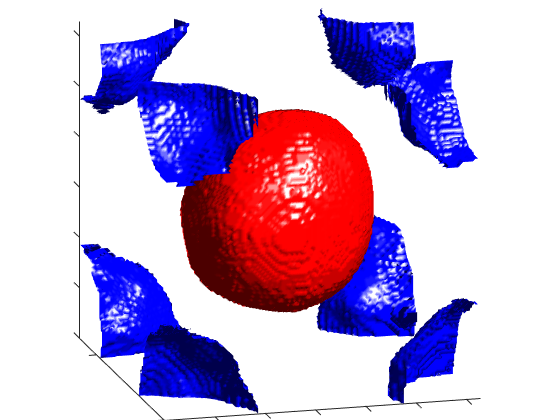}
\end{minipage}
\begin{minipage}{0.32\textwidth}
\centering \def\svgwidth{170pt}
\includegraphics[width=\textwidth]{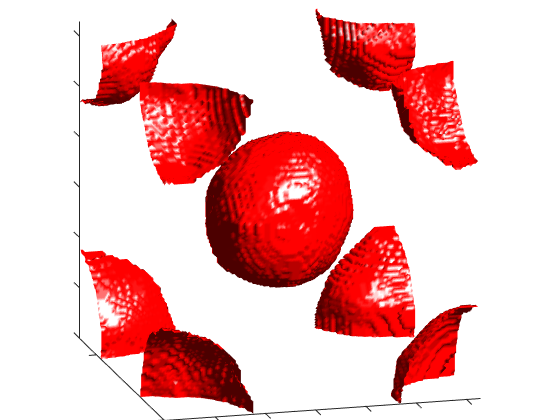}
\end{minipage}
\begin{minipage}{0.32\textwidth}
\centering \def\svgwidth{170pt}
\includegraphics[width=\textwidth]{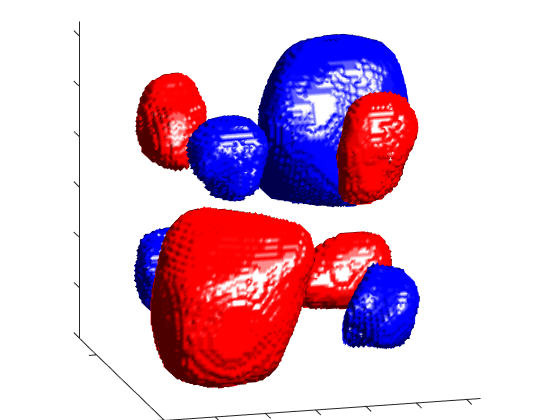}
\end{minipage}
\caption{Octuple gyre: Second to fourth left singular vectors at time $t_1=10.25$.}
\label{cs3D23left}
\end{figure}

\section{Conclusion and future directions}

We proposed to compute transfer operators in time-variant flows fields by a direct integration of the associated Fokker-Planck equation using spectral methods.  In particular, this approach does not require to integrate Lagrangian trajectories, which is particularly beneficial if the underlying flow field is only given by (a grid of) data. However, the spectral method described here is restricted to periodic domains. While it might be possible to treat cubical domains via pseudo spectral methods, a different approach for more complicated domains will be needed.

Further investigations will deal with the question of how to use other (spectral) bases such that the resulting transfer matrix is sparse, how to incorporate the information at intermediate times and how to apply the basic idea to other types of diffusion.

\section*{Acknowledgement}
The authors acknowledge the support by the DFG Collaborative Research Center SFB/TRR 109 "Discretization in Geometry and Dynamics". A.D. acknowledges the support by the "International Helmholtz Graduate School for Plasma Physics (HEPP)".

\appendix

\section*{Appendix}

\section{Proof of Lemma \ref{lemma:Pe_compact}}

\begin{proof}
We prove the lemma in the slightly more general case that $\divg(b)\neq 0$. $\|b\|_{C^0}:=\sup_{s\in [t_0,t_1], \ x \in X}\left\{|b(t,x)|<\infty,i=1,\ldots, d\right\}$. Let $u=u(t,x)$ be the solution of \eqref{FokkerPlanck} with initial condition $f=u_0\in {L}^2(X)$. Without loss of generality we set $\varepsilon=\sqrt{2}$ and $t_0=0$. We first note that $\|u\|_2$ is bounded by $\|u_0\|_2$:
\begin{equation}
\begin{split}
\frac{1}{2}\frac{d}{dt} \|u\|_2^2
& = \langle u, \partial_t u\rangle 
  = \langle u, \Delta u+\divg{(ub)}\rangle
  = \langle u, \Delta u\rangle + \langle u, \divg{(ub)}\rangle \\
&  = -\|\nabla u\|_2^2 - \langle \nabla u, ub\rangle
  = -\|\nabla u\|_2^2 - \langle \frac{1}{2} \nabla (u^2), b\rangle \\
&  = -\|\nabla u\|_2^2 -  \frac{1}{2} \langle u^2, \divg{(b)}\rangle
   \leq  -\|\nabla u\|_2^2 +  \frac{1}{2} \|u\|_2^2\|b\|_{C^1}.
\end{split}
\label{ineq1}
\end{equation}
Gronwall's inequality thus implies that for all $t>0$ 
\begin{align}
\|u\|_2^2\leq e^{t\|b\|_{C^1}}\|u_0\|_2^2.
\label{ineq2}
\end{align}
We now show that also $\|\nabla u\|_2$ is bounded by $\|u_0\|_2$. Because of \eqref{ineq1} we have
\begin{align*}
\|\nabla u\|_2^2\leq -\frac{1}{2} \frac{d}{dt}\|u\|_2^2+\frac{1}{2}\|u\|_2^2 \|b\|_{C^1}.
\end{align*}
Integrating from $t=t_0=0$ to $t=t_1$ we obtain
\begin{equation}
\begin{split}
\int_{0}^{t_1} \|\nabla u\|_2^2 \; dt
& \leq \frac{1}{2}\left( \|u_0\|_2^2 - \|u(t_1)\|_2^2\right) +\frac{1}{2}  \|b\|_{C^1} \int_{0}^{t_1} \|u\|_2^2\; dt\\
& \overset{\eqref{ineq2}}{\leq} \frac{1}{2} \|u_0\|_2^2+\frac{1}{2}  \|b\|_{C^1} \int_{0}^{t_1} e^{t\|b\|_{C^1}}\|u_0\|_2^2 ds\\
&= \frac{1}{2} e^{t_1\|b\|_{C^1}}\|u_0\|_2^2.
\end{split}
\label{ineq3}
\end{equation}
Therefore there is at least one $t^*\in [0,t_1]$, such that
\begin{align}
\|\nabla u(t^*,\cdot)\|_2^2 &\leq \frac{1}{2t_1} e^{t_1\|b\|_{C^1}}\|u_0\|_2^2,
\label{ineq4}
\end{align}
and we finally get
\begin{align*}
\frac{1}{2 }\frac{d}{dt} \|\nabla u\|_2^2 
& = \frac{1}{2} \frac{d}{dt}\langle \nabla u,\nabla u\rangle 
  = \frac{1}{2} \langle \frac{d}{dt}\nabla u, \nabla u\rangle+ \frac{1}{2} \langle \nabla u, \frac{d}{dt} \nabla u\rangle\\
&  = -\langle \Delta u, \partial_t u\rangle 
  = -\langle \Delta u,\Delta u+ \nabla (ub )\rangle\\
& = -\|\Delta u\|_2^2 - \langle \Delta u, \nabla u b\rangle -\langle \Delta u, u\divg{(b)}\rangle\\
&\overset{C-S}{\leq}  -\|\Delta u\|_2^2 +\|\Delta u\|_2\|\nabla u\|_2 \|b\|_{C^0}+\|\Delta u\|_2 \| u\|_2 \|b\|_{C^1}\\
&\leq  \frac{1}{2}\|\nabla u\|_2^2 \|b\|_{C^0}^2+ \frac{1}{2}  \| u\|_2^2 \|b\|_{C^1}^2,
\end{align*}
where we used that  $\frac{1}{2}(a^2c^2+d^2e^2)\geq-b^2+bac+bde$ for $a,b,c,d,e\in \mathbb{R}$.

Integration from $t=t^*$ to $t=t_1$ yields
\begin{equation}
\begin{split}
 \|\nabla u(t_1)\|_2^2
 &  \leq  \|\nabla u(t^*)\|_2^2 +\int_{t^*}^{t_1} \|\nabla u\|_2^2 \|b\|_{C^0}^2 \; dt + \int_{t^*}^{t_1}  \| u\|_2^2 \|b\|_{C^1}^2\; dt\\
& \hspace{-0.4cm}\overset{\eqref{ineq2}\eqref{ineq3} \eqref{ineq4}   }{\leq}  
\frac{1}{2t_1} e^{t_1\|b\|_{C^1}}\|u_0\|_2^2+ \|b\|_{C^0}^2 \frac{1}{2}e^{t_1\|b\|_{C^1}}\|u_0\|_2^2\\
& \phantom{\leq} + \|b\|_{C^1}^2 \int_{t^*}^{t_1} e^{t\|b\|_{C_1}} \|u_0\|_2^2 \;dt\\
& \leq\left(  \frac{1}{2t_1} + \frac{1}{2}\|b\|_{C^0}^2 + \|b\|_{C^1}  \right) \|u_0\|_2^2e^{t_1\|b\|_{C^1}}.
\end{split}
\label{ineq5}
\end{equation}
Combining equations \eqref{ineq2} and \eqref{ineq5} we obtain 
\begin{align*}
\|\mathcal{P}^\eps u_0 \|_{H^1}  =\|u(t_1,\cdot)\|_{H^1}\leq \left( 1+  \frac{1}{2t_1} + \frac{1}{2}\|b\|_{C^0}^2 + \|b\|_{C^1}  \right) e^{t_1\|b\|_{C^1}} \|u_0\|_2^2.
\end{align*}
Hence $\mathcal{P}^\eps$ maps bounded sets in ${L}^2(X)$ onto bounded sets in $H^1(X)$. As the embedding of $H^1(X)$ onto $L^2(X)$ is compact by Rellich's theorem, $\mathcal{P}^\eps$ is a compact operator. \hfill
\end{proof}

\newpage

\section{MATLAB code for the quadruple gyre example}

\begin{Code}[h]
\begin{lstlisting}[breaklines=true,numbers=left,linerange={1-46,50}]
% computation of coherent sets for a quadruple gyre system
%
%  Andreas Denner and Oliver Junge, TUM, 2015

M = 15; x = 2/M*(0:M-1)'; [X,Y] = meshgrid(x);     % collocation points
D = 1i*ifftshift(-(M-1)/2:(M-1)/2)';    % derivative in frequency space
Dy = D*ones(1,M); Dx = -ones(M,1)*D';
ep = 0.02; L = ep^2/2*(Dy.^2+Dx.^2);                 % Laplace operator

%% vector field
dl = 0.25; om = 2*pi;
f = @(t,x) dl*sin(om*t).*x.^2 + (1-2*dl*sin(om*t)).*x;
df = @(t,x) 2*dl*sin(om*t).*x + 1-2*dl*sin(om*t);
g = @(t,x,y) sin(pi*f(t,x)).*cos(pi*f(t,y)).*df(t,y);
w = @(t,v) -1/2*(Dx.*fft2(-g(t,X,Y).*ifft2(v))+Dy.*fft2(g(t,Y,X).*ifft2(v)) + fft2(-g(t,X,Y).*ifft2(Dx.*v))+fft2(g(t,Y,X).*ifft2(Dy.*v)));
v = @(t,x) reshape(w(t,reshape(x,M,M)),M^2,1);

%% intial values
N = 5;                                     % number of basis functions
U0 = zeros(M,M,M,M); 
for k = 1:M, for l = 1:M, U0(k,l,k,l) = 1; end, end
Y0 = U0(:,:,[1:(N-1)/2+1,(M-1)-((N-1)/2-2):M], ...
            [1:(N-1)/2+1,(M-1)-((N-1)/2-2):M]);

%% time integration / construction of transfer operator
t0 = 0; t1 = 10.25; m = 50; h = (t1-t0)/m; 
P = zeros(M^2,N^2);
for l = 1:N
    for k = 1:N
        y0 = reshape(Y0(:,:,k,l),M^2,1);
        P(:,N*(l-1)+k) = etdrk4(t0,m,h,L,v,y0);  
    end
end
[UU,S,VV] = svd(P); diag(S);      % compute singular values and vectors      

%% plot singular vectors
figure(1); clf; p = 128; [Xp,Yp] = meshgrid(2*(0:p-1)/p);
for j = 2
   v = reshape(VV(:,j),N,N)';  
   vp = zeros(p); 
   vp([1:(N+1)/2,p-(N-1)/2+1:p],[1:(N+1)/2,p-(N-1)/2+1:p]) = v*p^2/N^2;
   vp = ifft2(vp,'symmetric'); 
   %vp = vp/(sum(abs(vp(:,j)))/p^2*(2)^2);
   vp=vp/max(max(abs(vp)));    
   surf(Xp,Yp,vp), shading flat, view(90,90), caxis([-1 1]), axis equal, hold on, axis off, axis tight
 %  contour3(Xp,Yp,vp,[-0.5 -0.5],'linecolor',[0.95 0.95 0.95],'linewidth',4); 
%     unix(['convert -trim tmp.png ' filename]);
end

\end{lstlisting}

\end{Code}

\newpage

\bibliographystyle{abbrv}
\bibliography{refs}

\end{document}